\documentclass[11pt]{amsart}
\usepackage{amsmath, amsfonts,amsthm,amssymb,amscd, 
verbatim,graphicx,color,multirow,booktabs, caption,tikz,tikz-cd, mathdots,bm}

\usepackage[margin=2.5cm]{geometry}
\usepackage{tikz-cd}

\usepackage[breaklinks, colorlinks, citecolor=blue]{hyperref}
\hypersetup{colorlinks=true, linkcolor=red,linktocpage=True}

 \usepackage{graphicx}

\usepackage{latexsym}
\usepackage{longtable}
\usepackage{epsfig}
\usepackage{amsmath}
\usepackage{hhline}
\usepackage{comment}

\newcommand{\Alt}{\mathop{\mathrm{Alt}}}
\newcommand{\Sym}{\mathop{\mathrm{Sym}}}

\DeclareMathOperator{\BBF}{\mathbb{F}}

\DeclareMathOperator{\Out}{Out}
\newcommand{\Aut}{\mathop{\mathrm{Aut}}}

\DeclareMathOperator{\soc}{soc}
\DeclareMathOperator{\aut}{Aut}

\DeclareMathOperator{\GL}{GL} 
\DeclareMathOperator{\PSU}{PSU}
\DeclareMathOperator{\PGL}{PGL}
\DeclareMathOperator{\PSL}{PSL}

\DeclareMathOperator{\GU}{GU}

\DeclareMathOperator{\SU}{SU}
\DeclareMathOperator{\sym}{Sym}
 
\DeclareMathOperator{\AGL}{AGL}

\DeclareMathOperator{\alt}{Alt}

\renewcommand{\AA}{\mathcal A}

\renewcommand{\SS}{\mathcal S}
\renewcommand{\emptyset}{\varnothing}
\newcommand{\BB}{\mathcal{B}}

\newtheorem{thm}{Theorem}[section]
\newtheorem{cor}[thm]{Corollary}
\newtheorem{lemma}[thm]{Lemma}
\newtheorem{prop}[thm]{Proposition}

\theoremstyle{definition} 

\theoremstyle{definition}
\newtheorem{defn}[thm]{Definition}
\newtheorem{rem}[thm]{Remark}

\numberwithin{equation}{section}

\renewcommand{\footnote}{\endnote}
\newcommand{\ignore}[1]{}\makeglossary

\newcommand*{\medcap}{\mathbin{\scalebox{1.5}{\ensuremath{\cap}}}}%

\begin{document}
	\title[]{Primitive permutation IBIS groups}

\author[ Andrea Lucchini]{Andrea Lucchini}
\address{Andrea Lucchini, Dipartimento di Matematica “Tullio Levi-Civita", Università degli studi di Padova, Via Trieste 63,
35121, Padova, Italy   }
\email{lucchini@math.unipd.it}

\author[Marta Morigi]{Marta Morigi}
\address{Marta Morigi, Università di Bologna, Piazza di Porta San Donato 5, 40126 Bologna, Italy}
\email{marta.morigi@unibo.it}

\author[M. Moscatiello]{Mariapia Moscatiello}
\address{Mariapia Moscatiello, Dipartimento di Matematica, Università di Bologna, Piazza di Porta San Donato 5, 40126 Bologna, Italy} 
\email{mariapia.moscatiello@unibo.it}

\begin{abstract} Let $G$ be a finite permutation group on $\Omega$. An ordered sequence of elements of $\Omega$, $(\omega_1,\dots, \omega_t)$, is an irredundant base for $G$ if the pointwise stabilizer $G_{(\omega_1,\dots, \omega_t)}$ is trivial and no point is fixed by the stabilizer of its predecessors. If all irredundant bases of $G$ have the same size we say that $G$ is an IBIS group. In this paper we show that if a primitive permutation group is IBIS, then it must be almost simple, of affine-type, or of diagonal type. Moreover we prove that a diagonal-type primitive permutation groups is IBIS if and only if it is isomorphic to $\PSL(2,2^f)\times \PSL(2,2^f)$ for some $f\geq 2,$ in its  diagonal action of  degree $2^f(2^{2f}-1).$

\end{abstract}

\maketitle

\section{Introduction}

 Let $G\le \sym(\Omega)$ be a finite permutation group.  A subset $\BB$ of $\Omega$ is a \emph{base} for
  $G$ if the pointwise stabilizer $G_{(\BB)}$ is trivial and we denote by $b(G)$ the
minimal size of a base for $G$.   An ordered sequence of elements of $\Omega$, $\Sigma:=(\omega_1,\dots, \omega_t)$, is \emph{irredundant} for $G$ if no point in $\Sigma$ is fixed by the stabiliser of its predecessors. Moreover $\Sigma$ is an \emph{irredundant base} of $G$, if it is a base and it is irredundant. In particular, such an irredundant base provides the following stationary chain
$$G>G_{\omega_1}>G_{(\omega_1, \omega_2)}>\dots>G_{(\omega_1,\dots, \omega_{t-1})}>G_{(\omega_1,\dots, \omega_t)}=1,$$
where the inclusion of subgroups are strict.

Determining base sizes is a fundamental problem
in permutation group theory, with a long history stretching back to the nineteenth
century. Since the knowledge of how an element
$g\in G$ acts on a base $\BB$ completely determine the action of $g$ on
$\Omega$, bases plays an important role in computational group theory.
The smaller the base, the less memory is needed to store group elements, so there are practical reasons for trying to find small bases.

Irredundant bases for a permutation group possess some of the features of bases in a vector space.
Indeed, $\BB$ is a basis for a vector space $V$ if and only if it is an irredundant base for the general linear group $\GL(V)$ in its natural action on $V.$ 
However, some familiar properties of bases in vector spaces do not extend to bases for permutation groups: irredundant bases for groups in general are not preserved by re-ordering, and they can have different sizes.

In~\cite{ibis}, Cameron and Fon-Der-Flaass showed that all irredundant bases for a permutation group $G$ have the same size if and only if all the irredundant bases for $G$ are preserved by re-ordering. Groups satisfying one of the previous equivalent properties are called \emph{Irredundant Bases of Invariant Size groups}, \emph{IBIS groups} for short.
Moreover, Cameron and Fon-Der-Flaass~\cite{ibis} also proved that for a permutation group $G$ to be IBIS is a necessary and sufficient for the irredundant bases of $G$ to be the base of a combinatorial structure known with the name of matroid.  If this condition hold, then $G$ acts geometrically on the matroid and when $G$ acts primitively and is not cyclic of prime order, then the matroid is geometric (see~\cite{ibis} for more details). This brought Cameron to ask for a possible classification of the IBIS groups. As explained by Cameron himself in~\cite[Section 4-14]{cam}, there is no hope for a complete classification of IBIS group, when the cardinalities of the bases are large. But it might be reasonable to pose this question for primitive groups.

\

It is easy to observe that the symmetric group $\sym(k)$ in its natural action on $[k]:=\{1,\dots,k\}$ has base size $k-1$ and all subsets of $[k]$ with size $k-1$ are irredundant bases for $\sym(k).$ Similarly the base size for the action of the alternating group $\alt(k)$ acting naturally on $[k]$ has base size $k-2$ and all subsets of $[k]$ with size $k-2$ are irredundant bases for $\alt(k).$ There are only finitely many values of $k$ for which $\sym(k)$ or $\alt(k)$ acts as a primitive IBIS  group, not in its natural
representation ~\cite[Theorem 3.5]{ibis}. In particular if the action of $\sym(k)$ and $\alt(k)$  on $m$-sets of $[k]$ is IBIS, then $k\le 12$ and $m\le 5$. In this context, it is worth remembering that  $\sym(6)$ and $\alt(6)$ act as  primitive IBIS groups on 2-sets (with degree 15) and on partitions into $2$-sets of size $3$ of $[6]$ (with degree 10). There exist other almost simple primitive groups that are IBIS, \emph{e.g.} $\PSL(2,q)$ and $\PGL(2,q)$ on their action of degree $q+1$, $\PSL(2,2^f)$ acting on the cosets of a dihedral subgroup of order $2(2^f + 1)$, 
$\PSL(3,2)$ in degree 7, and $\alt(7)$ in degree 15. 
In another direction, it is not difficult to observe that Frobenius groups are IBIS group with base size $2$. Hence, the affine group $\AGL(1,p)$ is a primitive IBIS group with base size $2$, for every prime $p$. More in general, $\AGL(d,p)$ is a primitive IBIS group with base size $d+1$ for every prime $p$, and for every natural number $d$. 

\

Therefore we have quite a few examples of almost simple and affine-type primitive groups that are IBIS. So before launching in a complete classification of primitive IBIS groups, we decided to start the investigation of primitive permutation groups that can be IBIS in the classes of diagonal-type, product-type and twisted wreath product-type.
 The main result of this paper is the following theorem, that we prove in Section~\ref{sec:mainthm}.

\begin{thm}\label{thm:main} Let $G$ be a primitive permutation IBIS group. Then one of the following holds:
\begin{itemize}
\item[1.] The group $G$ is of affine type.
\item[2.] The group $G$ is almost simple.
\item[3.] The group $G$ is of diagonal type.
\end{itemize}
Moreover, $G$ is an IBIS primitive group of diagonal type if and only it belongs to the infinite family of non-monolithic diagonal type groups $\{\PSL(2,2^f)\times \PSL(2,2^f)\mid f\ge 2\}$ having degree $|\PSL(2,2^f)|=2^f(2^{2f}-1).$
\end{thm}

It is not difficult to see that there are no primitive permutation IBIS groups having non-abelian socle and base size $2$ (see Lemma~\ref{lem:notibis2}). This could lead to expect that there are not examples of IBIS groups with non-abelian socle and small base size. However, in Proposition~\ref{prop:sorpresa}, we prove that $\PSL(2,2^f)\times \PSL(2,2^f)$ in its diagonal action of degree $| \PSL(2,2^f)|$ is an IBIS group with base size $3$. In 
 Propositions~\ref{prop:solosorpresa0},~\ref{prop:solosorpresa} and Theorem~\ref{thm:nonmonolithicdiag}, as already mentioned in Theorem~\ref{thm:main}, we show that there are no other diagonal-type primitive permutation groups that are IBIS.
Investigating for this, we proved some results regarding non-abelian simple groups (see Subsection~\ref{subsec:PICP}) and about the base size of primitive diagonal-type permutation groups (Lemma~\ref{l1} and Lemma~\ref{lem:bPSL}). 
Finally, in Sections~\ref{sec:prod},~\ref{sec:TW}, we prove that  there are no IBIS primitive groups that are of product-type nor twisted wreath product groups.
%

\section{Preliminaries}
The modern key for analysing a finite primitive permutation group $G$ is to study the \textit{socle} $N$ of $G$. The socle of an arbitrary
finite group is isomorphic to the non-trivial direct product of simple
groups; moreover, for finite primitive groups these simple groups are
pairwise isomorphic. The O'Nan-Scott theorem describes in details the
embedding of $N$ in $G$ and collects some useful information about the
action of $G$. 
 In~\cite{LPSLPS}, Liebeck, Praeger and Saxl gave a self-contained proof, precisely five types of primitive groups
are defined (depending on the group and action structure of the
socle), namely the \textit{Affine-type}, the \textit{Almost Simple}, the \textit{Diagonal-type}, the  \textit{Product-type}, and the \textit{Twisted Wreath product}, and it is shown that  every primitive group belongs to exactly one of these types. We refer the reader to~\cite{LPSLPS} for the precise definition of the classes.

\

\noindent
Now, we state and prove same preliminary results regarding bases.

\begin{rem}\label{rem:ibisirredundantsize}
Since a base of $G$ with minimal size $b(G)$ is clearly irredundant, then a  permutation group $G$ is IBIS if and only if all its irredundant bases have size $b(G)$.
\end{rem}

%

	The proof of the following Lemma is straightforward.
\begin{lemma}\label{lem:partodadovevoglio} Let $G$ be a transitive permutation group on a set $\Omega$. Then $\BB\subset \Omega$ is a an irredundant base if and only if $\BB ^g\subset \Omega$ is an irredundant base for every $g\in G$. 
\end{lemma}

\begin{lemma}\label{lem:notibis2}
Let $G\leq \sym(\Omega)$ be a primitive group with non-abelian socle and let $M$ be a normal subgroup of $G$ containing the socle of $G.$ If $M$ is not regular, then $M$  has an irredundant base of size at least 3. In particular there are no IBIS primitive permutation groups $G$ with non-abelian socle and $b(G)=2.$
\end{lemma}
\begin{proof}
Let $\omega \in \Omega.$ The point stabilizer $G_\omega$ normalizes $M_\omega=M\cap G_\omega,$ hence $G_\omega \leq N_G(M_\omega).$
Since $M$ is not regular, $M_\omega$ is not normal in $G$, so
$G_\omega \leq N_G(M_\omega) <G.$ On the other hand $G_\omega$ is a maximal subgroup of $G$, so $G_\omega=N_G(M_\omega)$ and consequently $N_M(M_\omega)=M_\omega.$ Now assume by contradiction that $M_\omega\cap M_\omega^x=1,$ for every $x\in M\setminus M_\omega.$ Then $M$ would be a Frobenius group with kernel $K$ and complement $M_\omega$ (see \cite[8.5.5]{robinson}). In particular $K=\soc(M)$ would be a nilpotent normal subgroup of $G$ (see \cite[10.5.6]{robinson}), in contradiction with the fact that $\soc(G)$ is a direct product of isomorphic non-abelian simple groups.
\end{proof}

\subsection{Results on finite non-abelian simple groups}\label{subsec:PICP}

In this subsection, the finite non-abelian simple CT-group are characterized.

\begin{defn} 
A finite group $G$ is said to be a \emph{CT-group} group when commutativity is a transitive relation on $G\setminus \{1\}$.
\end{defn}

\begin{lemma}
A  group $G$ is a CT-group if and only if, for every $x, y\in G\setminus \{1\}$, either $C_G(x)=C_G(y)$ or $C_G(x)\cap C_G(y)=1.$
\end{lemma}
\begin{proof}
First, let assume that $G$ is a CT-group. Moreover assume that $t$ is a non-identity element of $C_G(x)\cap C_G(y).$ Since  $t$ commutes with both $x$ and $y$, then every element commuting with $x$ commutes with $t$ and consequently with $y$. Similarly every element commuting with $y$ commutes with $t$ and consequently with $x$. That is $C_G(x)= C_G(y).$

Conversely assume that for every $x, y\in G\setminus \{1\}$, either $C_G(x)=C_G(y)$ or $C_G(x)\cap C_G(y)=1.$ Let $a,b, c\in G\setminus \{1\}$ and assume that $a$ commutes with $b$ and $b$ commutes with $c$. Then $a\in C_G(a)\cap C_G(b)$ and $b\in  C_G(b)\cap C_G(c),$ so $C_G(a)= C_G(b)=C_G(c)$.
\end{proof}

\begin{lemma}\label{lem:PICP} 
The group $T=\PSL_2(2^f)$, with $f\ge 2$ is a CT-group.
\end{lemma} 
\begin{proof}
Let $q=2^f.$ The subgroups of $\PSL_2(q)$ are known by a theorem of Dickson. A complete list of all subgroups of $\PSL_2(2^f)$ can be found in \cite{huppert}. We have the following possibilities:
\begin{enumerate}
\item Elementary-abelian $2$-groups.
\item Cyclic groups of order $z$, where $z$ divides $q\pm 1$. 
\item Dihedral subgroups of order $2z$, where $z$ divides $q\pm 1$. 
\item Alternating groups $\alt(4)$ if $f\equiv 0$ (mod 2).
\item Alternating groups $\alt(5)$ if  $q^2-1 \equiv 0$ (mod 5).
\item Semidirect products $C_2^m \rtimes C_t$ of elementary-abelian groups of order $2^m$ with cyclic groups of
order $t$, where $t$ divides $2^m -1$.
\item Groups $\PSL_2(2^m)$ if $m$ divides $f$ and $\PGL_2(2^m)$ if $2m$ divides $f$.
\end{enumerate}
Take $x\in T\setminus \{1\}$, then $x$ is a non-trivial central element of $C_T(x).$ Hence, the list of all subgroups of $\PSL_2(2^f)$ yields that $C_T(x)$ is an abelian subgroup of $\PSL_2(2^f)$.
Now let $a, b,$ and $t$ be non-trivial elements of $T$ such that $a$ commutes with $t$ and $t$ commutes with $b.$ Since $a$ and $b$ are elements of $C_T(t)$ that is an abelian group, then $a$ and $b$ commute.
\end{proof}

%
%
%
%

In 1957, Suzuki~\cite{Suz}\ showed that a finite non-abelian simple CT-groups is isomorphic to some $\PSL(2,2^f)$ with $f\ge 2$. Hence the following result holds true.
\begin{thm}\label{thm:CT} A finite non-abelian simple group is a CT-group if and
 only if it is isomorphic to $\PSL(2,2^f)$ for some $f\ge 2$.
\end{thm}

\section{Diagonal type}

Let $k\ge 2$ be an integer, and let $T$ be a non-abelian simple group. We set $[k]:=\{1,\dots,k\}$ and 
 \begin{align*}
W(k,T)&:=\{(x_1,...,x_k)\pi \in \Aut(T)\wr \sym(k) \mid x_1\equiv x_i\!\! \mod T\;
\mbox{for all }i\in [k] \},\\
D(k,T)&:=\{(x,\dots, x)\pi \in  \Aut(T)\wr \sym(k) \},\\
\Omega(k,T)&:= D(k,T) / W(k,T),\, \mbox{the set of the right cosets of }D(k,T) \mbox{ in }W(k,T) \\
A(k,T)&:= W(k,T) \cap \Aut(T)^k.
\end{align*}
Observe that $W(k,T) = A(k,T)\rtimes \sym(k)= T^k . (\Out(T) \times \sym(k))$.
The group $G$ is of diagonal type if there exists an integer $k\ge 2$ and a non-abelian simple group $T$ such that $T^k \le G\le W(k,T)$ and $G$ acts primitively on $\Omega(k,T)$. In the whole section we will assume that $G$ is a primitive permutation group of diagonal type. Note that $G$ has socle $T^k$ and degree $|T|^{k-1}$. Recall that $G$ is primitive on $\Omega(k,T)$ when 
$$P_G:=\{\pi \in \sym(k) \mid (g_1,\dots,g_k)\pi\in G\; \mbox{for some }(g_1,\dots,g_k) \in A(k,T) \}\le \sym(k)$$ is primitive on $[k]$ or $k=2$ and $P_G=1.$ If $k=2$ and $P_G=1$ then $G$ has two minimal normal subgroups, hence the group is non-monolithic. In all other cases, the primitive diagonal-type groups are monolithic.

Throughout the section, we shall use without further comments the following notation. 
Given $t, s_2,\dots, s_k\in T$, we set
$$H:=\{x \in \aut(T) \mid (x,\dots, x)\in G\},\; \;\;C(t):=C_H(t),\;\;\mbox{and}\;\; [1,s_2,\dots, s_{k}]:=D(k,T)(1,s_2,\dots, s_{k}).\\
$$

From Remark~\ref{rem:ibisirredundantsize}, it is clear that the study of primitive permutation IBIS groups is intimately related to the base sizes.
For primitive diagonal groups, the minimal base size has been studied in \cite{diagonal}. In the following proposition,  we summarize the results proved in \cite{diagonal} which we shall use in this section.

\begin{prop} \cite{diagonal}
\label{basediag}  The following hold true.
\begin{enumerate}
\item  If the top group $P_G$ does not contain $\alt(k)$, then $b(G) =2$.
\item If $k=2$, then $b(G)=3$ when $P_G =1$, and $b(G)\in \{3,4\}$ otherwise.
\item If $P_G = \Alt(k)$, then $b(G) = 3$ when $k = 2$, and $b(G) = 2$ when $k$ is $3$ or $4$. 
\item If $P_G = \Sym(k),$ then
$b(G)\in\{3, 4\}$ when $k = 2$, and $b(G)\in\{2, 3\}$ when $k$ is $3$ or $4$.
\end{enumerate}
\end{prop}


\subsection{Diagonal non-monolithic primitive groups} In this subsection, $k=2$ and $P_G=1.$

\medskip 

Here we are going to prove that $G$ is IBIS if and only if 
$G \cong (\PSL(2,2^f))^2,$ for some $f\ge 2$.

\medskip 

It is not difficult to observe that the sequence $([1,1],[1,t_1], [1,t_2],\dots, [1,t_r] )$ is an irredundant sequence for $G$ if and only if the following hold true:
\begin{itemize}
\item $t_1,\dots, t_r$ are non trivial elements of $T;$ 
\item for $1\le i\le r-1$, the subgroup $\underset{1\le j\le i}{\medcap} C(t_j)$ properly contains $ \underset{1\le j \le i+1}{\medcap} C(t_j)$
\item  $ \underset{1\le j \le r}{\medcap} C(t_j)=1.$
\end{itemize}

\begin{prop}\label{prop:solosorpresa0}
	If $k=2$, $P_G=1$  and $G$ is an IBIS group, then $T=(\PSL(2,2^f))^2,$ for some $f\ge 2$.
\end{prop}
\begin{proof}
If $\soc(G)$ is not isomorphic to $(\PSL(2,2^f))^2,$ then by Theorem~\ref{thm:CT} we can choose two non trivial elements of $T$, $s$ and $t$, such that $C(t)\cap C(s)\ne 1$ and $C(t)\neq C(s).$ Then $([1,1], [1,s], [1,t])$ is an irredundant sequence that is not a base. In particular we can complete $([1,1], [1,s], [1,t])$ to an irredundant base of $G$, thus there exists an irredundant base of $G$ with size at least $4$.  Since $b(G)=3$ by Proposition \ref{basediag}, it follows from Remark~\ref{rem:ibisirredundantsize} that $G$ is not IBIS.
Therefore we conclude that $\soc(G)=T^2=(\PSL(2,2^f))^2.$
\end{proof}

\begin{prop}If $k=2$, $P_G=1$, $T=\PSL(2,2^f)$ and $G$ is IBIS, then $G=\soc(G)=T^2.$

\end{prop}
\begin{proof}\label{prop:solosorpresa}
Let $\varphi\in \Aut(T)$ be the automorphism induced by the Frobenius automorphism of the field $\BBF_{2^f}$, so that $\Aut(T)=\langle\varphi\rangle T$. If $G\ne T^2$, then there exists $\varphi^i \in \Aut(T)$ such that $(\varphi^i,\varphi^i)\in G,$ that is  $\varphi^i \in H.$
Consider the following elements of $T$:
$$w:=\begin{pmatrix} 0&1\\1&0\end{pmatrix}\,\;\mbox{and}\, \;z:=\begin{pmatrix} 1&1\\0&1\end{pmatrix}.$$
Note that $w\in C(w)\setminus  C(z).$
Moreover, it is clear that $w^{\varphi^{i}}=w$ and $z^{\varphi^{i}}=z,$ that is $\varphi^i\in C(w)\cap  C(z)$. Hence $([1,1],[1,w],[1,z])$ is an irredundant sequence that is not a base. Hence there exists an irredundant base of $G$ with size at least $4$.  Since $b(G)=3$ by Proposition \ref{basediag}, it follows from Remark~\ref{rem:ibisirredundantsize} that $G$ is not IBIS.
Therefore we conclude that $G=T^2=(\PSL(2,2^f))^2.$
\end{proof}

\begin{prop}\label{prop:sorpresa}
Let $G=T^2=(\PSL_2(2^f))^2,$ where $f\ge2$. Then $G$ is an IBIS group.
\end{prop}
\begin{proof} Since $b(G)=3$  by Proposition \ref{basediag}, it suffices to show that each irredundant base has size 3. Denoting by $\BB$ an irredundant base,
by Lemma~\ref{lem:partodadovevoglio} we can assume that $[1,1]$ is the first element of $\BB.$ Since $|\BB|\ge 3,$ there exist  two different non-trivial elements $s$ and $t$ of $T$, such that $[1,t]$ and $[1,s]$ are in $\BB.$ 
Moreover, since $\BB$ is irredundant, then $C_T(t)\ne C_T(s)$. Hence from Theorem~\ref{thm:CT}, we deduce that $C_T(t)\cap C_T(s)=1$ and so $\BB$ has size 3.
\end{proof}

\

\subsection{Diagonal monolithic primitive groups}
Here we are going to show that the monolithic primitive groups of diagonal type are not IBIS group. Recall that if a group $G$ of diagonal type with $T^k\le G\le  W(k,T)$ is monolithic, then $P_G$ is primitive on its action on $k$ points. Throughout the subsection, we shall assume that $G$ is monolithic without further mentioning.

\

First we assume $k=2.$ In this case $T^2 \le G\le W(2,T)$
and $P_G=\langle\sigma\rangle,$ where $\sigma:=(1,2).$ 
Note that there exists $y\in \aut(T)$ such that $(y,y)\sigma \in G.$ Recall that $H=\{z\in\Aut(T)|(z,z)\in G\}$ and consider the coset $Hy.$ Since $((y,y)\sigma)^2=(y^2,y^2)$, we have that $y^2\in H.$ For any $x\in \aut(T)$,  we get that $(x,x)\sigma\in G$ if and only if $Hx=Hy.$ Further, notice that $Hy=H$ if and only if $\sigma \in G.$ 

For any $t \in T$, define
$$I(t):=\{x\in yH\mid x^{-1}tx=t^{-1} \}.$$

\begin{lemma}
Let $t_1,\dots,t_r\in T$ and $\BB=([1,1],[1,t_1],\dots,[1,t_r])$.
Then $\BB$ is a base for $G$ if and only if  $$\underset{1\le i\le r}{\medcap} C(t_i)=1 \text { and }\underset{1\le i\le r}{\medcap}I(t_i)=\emptyset.$$
Moreover,
$\BB$ is an irredundant sequence if and only if $t_1,\dots,t_r $ are non-trivial elements of $T$ and, for $1\le j\le r-1$, either the set $\underset{1\le i\le j}{\medcap} C(t_i)$  properly contains  $\underset{1\le i\le j+1}{\medcap} C(t_i)$ or the set $\underset{1\le i\le j}{\medcap} I(t_i)$  properly contains  $\underset{1\le i\le j+1}{\medcap} I(t_i)$. 
\end{lemma}
\begin{proof}We have $G_{[1,1]}=\{(h,h), (x,x)\sigma \mid h\in H, x\in yH\}.$ To conclude it is suffices to notice
that $(h,h)\in G_{[1,t_i]}$ if and only if $h\in  C(t_i),$ while $(x,x)\sigma \in   G_{[1,t_i]}$ if and only if $x\in I(t_i).$
\end{proof}
The following result will play an important role in investigating the case $k=2$.

\begin{thm}\cite[Theorem 1.1]{LeLi}\label{b3} Let $T$ be a non-abelian finite simple group, that is not $\Alt(7),$ $\PSL(2,q),$
$\PSL(3,q)$ nor $\PSU(3,q).$ Then there exist $x, t\in T$ such that the following hold:
\begin{itemize}
\item[(i)] $T = \langle x, t\rangle;$
\item[(ii)] $t$ is an involution;
\item[(iii)] there is no involution $\alpha\in\Aut(T)$ such that $x^\alpha=x^{-1},$ $t^\alpha=t.$
\end{itemize}
\end{thm}

%

\begin{lemma}\label{l1}
Assume that $T^2 \le G\le W(2,T)$, with $T$ as in Theorem \ref{b3}. Then $b(G)=3.$  
\end{lemma}
\begin{proof}
Let $T=\langle x,t\rangle,$ with $x,t$ as in Theorem \ref{b3} and let $\BB:=([1,1], [1,x], [1,t])$.
Note that $C_{\Aut(T)}(x)\cap C_{\Aut(T)}(t)=1$ and $I(x)\cap I(t)=\emptyset,$ hence $\BB$ is a base for $G$, that is $b(G)\le 3.$ 
By Proposition \ref{basediag}, $b(G)\ge 3,$ so the result follows.
\end{proof}
%
%
%

\begin{prop}\label{p1} 
Assume that $T^2 \le G\le W(2,T)$, with $T\neq \PSL(2,2^f)$ for every $f\ge 2$.
If $b(G)\le 3$, then $G$ is not an IBIS group.
\end{prop}
\begin{proof}
By Theorem~\ref{thm:CT}, there exist two nonidentity elements $s$ and $t$ in $T$ such that $C_T(t)\cap C_T(s)\ne 1$  and  $C_T(t)\neq C_T(s).$ Then $([1,1], [1,s], [1,t])$ is an irredundant sequence that is not a base. In particular we can complete $([1,1], [1,s], [1,t])$ to an irredundant base of $G$, thus there exists an irredundant base of $G$ with size at least $4$.  Since we are assuming $b(G)\le 3,$ the result  follows from Remark~\ref{rem:ibisirredundantsize}.
\end{proof}

Combining Lemma \ref{l1} and Proposition \ref{p1} we get the following:

\begin{cor} \label{cor:generatoribelli} Let $T^2 \le G\le W(2,T)$, with $T$  as in Theorem \ref{b3}. Then  $G$ is not an IBIS group.\end{cor}

\

In what follows, we are going to show that also when $T\in\{ \Alt(7),\PSL(2,q),\PSL(3,q),\PSU(3,q)\}$, the group  $T^2 \le G\le W(2,T)$ is not IBIS. 

\begin{prop}\label{prop:A7}
 If $(\alt(7))^2\le G\le W(2,\alt(7)),$ then $G$ is not an IBIS group.\end{prop}
\begin{proof}

If 
$(1,2)\in yH$ then $((1,2),(1,2))\sigma\in G.$ In this case, 
consider $s_1=(1,2)\,(3,4),$ $s_2=(1,2)\,(3,5)$ and $s_3=(1,2)\,(3,6).$ Since
 $(1,2)\in I(s_1)\cap I(s_2) \cap I(s_3)$ and $C_{\alt(7)}(\langle s_1,s_2,s_3\rangle) < C_{\alt(7)}(\langle s_1,s_2\rangle < C_{\alt(7)}(s_1)),$ then $([1,1], [1,s_1],[1,s_2], [1,s_3])$ is an irredundant sequence of $G$ that is not a base.
%

Notice that $(1,2)\notin yH$ if and only if $H=\alt(7)= Hy.$
In this case, consider $t_1=(1,2,3),$ $t_2=(1,2,4),$ and $t_3=(1,2,5).$ Since
$(1,2)(6,7)\in I(t_1)\cap I(t_2) \cap I(t_3)$ and $C_{\alt(7)}(\langle t_1,t_2,t_3\rangle)< C_{\alt(7)}(\langle t_1,t_2\rangle)<  C_{\alt(7)}(t_1),$ then $([1,1], [1,t_1],[1,t_2], [1,t_3])$ is an irredundant sequence of $G$ that is not a base.

  In particular, in both cases, completing the irredundant sequences to a base of $G$, we find an irredundant base of $G$ with size at least $5$. By Proposition~\ref{basediag} $b(G)\le 4,$ hence the result follows from Remark~\ref{rem:ibisirredundantsize}.
\end{proof}

For $T=\PSL(2,q)$ we need to work differently according to the parity of $q$. Since $\PSL(2,4)\cong \PSL(2,5)\cong\alt(5),$ in order to simplify the investigation, we deal with $T=\alt(5)$ saparately.

\begin{prop}\label{prop:A5} Let $T=\alt(5)$ and let $T\le G\le W(T,2).$ Then  $G$ is not an IBIS group.
\end{prop} 
\begin{proof}
We prove the result by considering different cases and, in each of them, we set $a:=(1,2,3).$

First, let $yH\ne H.$ 
Under this assumption we have that $H=T=\alt(5)$ and so, without loss of generality, we can assume that $y=(1,2).$
Let $a_2:=(1,2,3,4,5)$ and let $\BB_1:=([1,1], [1, a], [1, a_2]).$ Note that $a\in C(a)\setminus C(a_2),$ $ C(a)\cap C(a_2)=1,$ and $I(a)\cap I(a_2)=\emptyset.$ Hence $\BB_1$ is an irredundant base of size $3$.
Now, let $b_2:=(1,2,4)$. Note that $a\in C(a)\setminus C(b_2)$ and $(1,2)\in I(a)\cap I(b_2)$. Hence $([1,1], [1, a], [1, b_2])$ is an irredundant sequence that is not a base. Completing this to an irredundant bases of $G$, we produce an irredundant base of size at least $4$. By Remark~\ref{rem:ibisirredundantsize}, $G$ is not IBIS in this case.

Let $yH=H=T$, let $c_2:=(1,2)(3,4)$, and let $\SS_1:=([1,1], [1, a], [1, c_2 ])$.  Since 
$C(a)\cap C(c_2)=1$  and $I(a)\cap I(c_2)=\emptyset,$ then $\SS_1$ is an irredundant base of size $3$.
Let $d_2:=(1,2)(4,5).$ Then $a\in C(a) \setminus C(d_2)$ and $d_2\in I(a)\cap  I(d_2)$. Consequently $([1,1], [1, a], [1, d_2])$  is an irredudant sequence that is not a base and we can conclude as in the previous paragraph.

Finally, we need to consider the case $yH=H=\sym(5)$. Here, 
let $t:=(3,4,5)$, $s_1=(1,2)\,(3,4),$ and $s_2:=(1,2)\,(4,5).$ Since $t\in C(t)\setminus C(s_1),$ $s_1\in (I(t)\cap I(s_1))\setminus I(s_2),$ and
 $(1,2)\in C(t)\cap C(s_1) \cap C(s_2)$, then $([1,1], [1,t],[1,s_1], [1,s_2])$ is an irredundant sequence of $G$ that is not a base.
Thus  completing the irredundant sequences to a base of $G$, we find an irredundant base of $G$ with size at least $5$. By Proposition~\ref{basediag} $b(G)\le 4,$ hence the result follows from Remark~\ref{rem:ibisirredundantsize}.
\end{proof}

\begin{prop}\label{bPSL2q} Let $T=\PSL(2,q)$. If $q>5$ is odd and $T^2 \le G\le W(2,T)$, then $G$ is not an IBIS group.
\end{prop}
\begin{proof}
The conjugacy classes of subgroups of $T=PSL(2,q)$ are well-known (see for instance~\cite{huppert}). Among them there are a conjugacy class of dihedral groups of order $q-1$ (in  Aschbacher class ${\mathcal C}_2$) and a conjugacy class of dihedral groups of order $q+1$
(in  Aschbacher class ${\mathcal C}_3$).  We can choose $\epsilon\in\{\pm 1\}$ such that $a=\frac{q+\epsilon}2$ is odd. Let $b=\frac{q-\epsilon}2.$  In $T$ there is just one conjugacy class of involutions and the centralizer of an involution is a dihedral group $D_{2b}$ of order $2b$. Note that $|T|=\frac 12q(q^2-1)=2qab.$ Thus the total number of involutions is $qa$. There are $qb$ subgroups isomorphic to a dihedral group $D_{2a}$ of order $2a$ and each of them contains $a$ involutions. Hence every involution is contained in $b$ different subgroups of type $D_{2a}$.
In particular, there exist two different  subgroups $M_1$ and $M_2,$ both isomorphic to $D_{2a}$, such that $M_1\cap M_2$ contains an involution $x$. 

First, assume $q > 11.$ In this case the dihedral groups of order $2a$ are maximal in $G$.
Let
$y_i\in M_i$ be an element of order $a$ such that $M_i=\langle y_i,x\rangle$ and $y_i^x=y_i^{-1}$, for $i=1,2.$ Since there is no maximal subgroup of $T$ containing both $y_1$ and $y_2$, we have that $T=\langle y_1,y_2\rangle.$ The previous argument does not work if $q\in \{7, 9, 11\}$ but it can be easily checked that, also in these cases, $T$ contains two elements $y_1, y_2$ of order $a$ and an involution $x$ such that $y_1^x=y_1^{-1}, y_2^x=y_2^{-1}$ and $\langle y_1, y_2\rangle=T.$

Note that if $z\in \Aut(T)$ is such that $y_i^z=y_i^{-1}$, for $i=1,2,$ then $zx^{-1} \in C_{\aut(T)}(\langle y_1, y_2\rangle)=1,$ hence $z=x.$

Assume that $x\not\in yH.$ Then we have 
that $I(y_1)\cap I(y_2)=\emptyset$ and $C(y_1)\cap C(y_2)=1$,  thus $\BB=([1,1], [1,y_1], [1,y_2])$ is an irredundant base for $G$ of size $3$. Hence, by Proposition \ref{p1}, $G$ is not IBIS.

Now, assume that  $x\in yH,$ so that $y\in Hx=H$. 
Note that there exists a subgroup of type $D_{2b}$ containing $x$ and $r$ such that $r$ has order $b$ and $r^x=r^{-1}.$ From $q\neq 5$, it follows $b>2$ and $r\neq r^{-1}.$ Moreover we have that $T=\langle x, y_1,r\rangle$, $x\in C(x)\setminus C(\langle x,r\rangle)$, $r^{\frac b2}\in C(\langle x,r\rangle )\setminus C(\langle x,r,y_1\rangle),$ and $x\in I(x)\cap I(r)\cap I(y_1).$ Hence $([1,1], [1,x], [1,r],[1,y_1])$ is an irredundant sequence that is not a base, so $G$ has an irredundant base of size at least 5. As usual, the result follows from Proposition~\ref{basediag} and Remark~\ref{rem:ibisirredundantsize}.
\end{proof}

%

\begin{lemma}\label{intersecodiedrali}
	Let $T=\PSL(2,2^f)$ with $f\geq 3.$
	Then there exists two maximal subgroups $K_1, K_2$ of $T$ such that $K_1\cong K_2 \cong D_{2(q-1)}$ and $|K_1\cap K_2|=1.$
\end{lemma}
\begin{proof}
Let $q=2^f.$ In $T$ there is just one conjugacy class of involutions and the centralizer of an involution is a Sylow 2-subgroup of order $q$. Note that $|T|=q(q^2-1).$ Thus the total number of involutions is $q^2-1$. There are $q(q+1)/2$ maximal subgroups isomorphic to  $D_{2(q-1)}$ (all conjugated in $T$). Denote by $\Omega$ the set of these sugroups. Any element of $\Omega$  contains $q-1$ involutions. Hence every involution is contained in $q/2$ different subgroups of $\Omega.$ Now, fix $D\in \Omega.$ If $K\in \Omega$ and $K\neq D,$ then $|D\cap K|\leq 2.$ Let
$\Omega_D=\{K \in \Omega \mid |D \cap K|=2\}.$ For any involution $z \in D$, there are exactly $q/2-1$ subgroups $K\in \Omega_D$ with $D\cap K=\langle z\rangle,$ hence
$|\Omega_D|=(q-1)(q/2-1)<|\Omega|-1=q(q+1)/2-1$. Thus there exists $K\neq D$ with $K \in \Omega \setminus \Omega_D.$
\end{proof}

\begin{lemma}\label{lem:bPSL} Assume that $T^2 \le G\le W(2,T)$, with $T=\PSL(2,2^f)$ and $f\ge 3$.
If either $yH\ne H$ or $yH=H=T$, then $b(G)=3$.\end{lemma}
\begin{proof}  
First, we assume that $yH\ne H$. Since $q=2^f$, we have that $\PSL(2,q)$ is isomorphic to ${\mathrm{SL}}(2,q)$,  so we may deal with $T={\mathrm{SL}}(2,q).$
 Let $a$ be an element of multiplicative order $q-1$ in the field $\BBF_q$. 
Consider the following elements of $T:$
$$x:=\begin{pmatrix} a&0\\0&a^{-1}\end{pmatrix},\;w:=\begin{pmatrix} 0&1\\1&0\end{pmatrix},\;\mbox{and}\,\;z:=wx=\begin{pmatrix} 0&a^{-1}\\a&0\end{pmatrix}. $$
Then $N:=\langle w,z\rangle$ is a dihedral group of order $2(q-1)$. As $C_T(x)=\langle x\rangle$, it follows that $C_T(N)=1$.

 We are going to show that $C_{\Aut(S)}(N)=1.$ Denote  by $\varphi\in \Aut(T)$ the automorphism induced by the Frobenius automorphism of the field $\BBF_q$. 
Let $s\in T$, let $1\le i\le f,$ and let assume that $\varphi^is$ centralizes $N$. 
%
In particular, $\varphi^is$ centralizes $x$, and so $x^{\varphi^i}=x^{s^{-1}}$ and $x$ are conjugated. Therefore $\{a, a^{-1}\}=\{a^{2^i}, a^{-2^i}\}$.

If $a=a^{-2^i}$, then $a^{2^i+1}=1$ and so $q-1=2^f-1$ divides $2^i+1$, contradicting the assumption $f\ge 3$. 
Hence $a=a^{2^i}$, that is $\varphi^i$ is the identity. Consequently $s\in C_T(N)=1.$
This proves that $C_{\Aut(S)}(N)=1$ and in particular that $C(z)\cap C(w)=1$.

If $I(w)\cap I(z)\ne\emptyset$ , then there exist $i\in \mathbb{N}$ and $s\in T$ such that $\varphi^is\in I(w)\cap I(z)$.  Since $w$ and $z$ are involutions, then $\varphi^is\in C_{\Aut(S)}(N)=1$. Consequently $1\in yH$, that contradicts $yH\ne H$.

Hence $I(w)\cap I(z)=\emptyset$ and then $([1,1], [1,z], [1,w])$ is a base for $G$. Combining this and Proposition~\ref{basediag}, we deduce that $b(G)=3$.

\medskip

Now, assume that $yH=H=T.$ By Lemma \ref{intersecodiedrali}, $G$ contains two subgroups $H_1$ and $H_2$ isomorphic to a dihedral group of order $2(q-1),$ with $H_1\cap H_2=1.$ Take $x_1 \in H_1, x_2 \in H_2,$ both of order $q-1.$ There is no maximal subgroup containing both $x_1$ and $x_2$, so $T=\langle x_1,x_2 \rangle.$ In particular, $C(x_1)\cap C(x_2)=1$. Moreover, if $z$ is an involution inverting $x_1,$ then $z\in H_1$. Hence, there are no involutions in $T$ inverting both $x_1$ and $x_2.$
Therefore $([1,1], [1,x_1], [1,x_2])$ is an irredundant base of size 3. Combining this and Proposition~\ref{basediag}, we deduce that $b(G)=3$.
\end{proof}


\begin{prop}\label{PSL}Let $T=\PSL(2,2^f)$, with $f\ge 3.$ If $T^2 \le G\le W(2,T)$ then $G$ is not an IBIS group.\end{prop}
\begin{proof} As above, we work in $T={\mathrm{SL}}(2,2^f)$.
First, assume that $yH\ne H$. Without loss of generality, denoting again by $\varphi$ the automorphism induced by the Frobenius automorphism of the field $\BBF_{2^f}$, we can say that $y=\varphi^{t},$ for some $t\in \mathbb{N}.$
Since $y^2\in H$, then $f=2r$ is even. Moreover, we can assume that $t$ divides $r.$
Since $2^{t}+1$ divides $2^{f}-1,$
there exists a nonzero element $c\in\BBF_{2^f}$ of multiplicative order $2^t+1$.
 Consider the following elements
$$x:=\begin{pmatrix} c&0\\0&c^{-1}\end{pmatrix}\; \mbox{and}\;w:=\begin{pmatrix} 0&1\\1&0\end{pmatrix}.$$
Note that $x\in C(x)\setminus C(\langle x,w\rangle )$. Further, $x^{y}=x^{\varphi^{t}}=x^{-1}$ and
$w^{y}=w^{\varphi^{t}}=w=w^{-1}$, so that $y\in I(x)\cap I(w)\ne\emptyset$. Hence $([1,1], [1,x], [1,w])$ is an irredundant sequence that is not a base. From Lemma~\ref{lem:bPSL}, we have that $b(G)=3$, hence the result follows from Remark~\ref{rem:ibisirredundantsize}.

\medskip

Therefore we can assume that $yH=H$, so that $\sigma\in G$.

Let assume that there exists $1\neq \varphi^i \in \Aut(T)$ such that $(\varphi^i,\varphi^i)\in G,$ that is $H>T.$
Consider the following elements of $T:$
$$x:=\begin{pmatrix} 1&1\\1&0\end{pmatrix},\ \;w:=\begin{pmatrix} 0&1\\1&0\end{pmatrix}\;\mbox{and} \, \;z:=\begin{pmatrix} 1&1\\0&1\end{pmatrix}.$$
Note that $N:=\langle x, w \rangle$ is a dihedral group of order $2|x|$ and $w$ inverts $x$. In particular $x\in C(x)\setminus  C(w),$ $w \in I(x)\cap I(w),$ and $w\notin I(z)$.
Moreover, it is clear that $x^{\varphi^{i}}=x$, $w^{\varphi^{i}}=w,$ and $z^{\varphi^{i}}=z$, that is $\varphi^i\in C(x)\cap  C(w)\cap C(z)$. Hence $([1,1],[1,x],[1,w],[1,z])$ is an irredundant sequence that is not a base. This produces a base of $G$ with size at least $5$. Since $b(G)\leq 4$ by Proposition \ref{basediag}, the result follows from Remark~\ref{rem:ibisirredundantsize}.

Finally, assume $Hy=H=T.$ Notice that $T$ contains two involutions $z_1, z_2$ generating a dihedral group of order $2(2^f+1).$ Since $C(z_1)\cap C(z_2)=I(z_1)\cap I(z_2)=1,$ 
for every $x \in T$ with $|x|>2,$ $([1,1], [1,z_1], [1,z_2], [1,x])$ is an irredundant base of size 4. Now, the result follows combining~\ref{lem:bPSL} and Remark~\ref{rem:ibisirredundantsize}.
%
%
%
\end{proof}

\begin{prop}\label{PSL3q} If $T=\PSL(3,q)$ and $T^2 \le G\le W(2,T)$, then $G$ is not an IBIS group.\end{prop}
\begin{proof} Let $q=p^f$, let
$$N:=\left\{ \begin{pmatrix} 1&x&y\\0&1&z\\0&0&1\end{pmatrix}\big| x,y,z\in \BBF_p \right\}\subseteq {\mathrm{SL}}(3,q),$$ and consider the following elements of $N:$
$$a:=\begin{pmatrix} 1&1&0\\0&1&0\\0&0&1\end{pmatrix},\, b:=\begin{pmatrix} 1&0&0\\0&1&1\\0&0&1\end{pmatrix},\,c:=\begin{pmatrix} 1&0&1\\0&1&0\\0&0&1\end{pmatrix}.$$
Since $N$ intersects the center of ${\mathrm{SL}}(3,q)$  trivially, we may identify $N$ with a subgroup of $T.$ Note that $N$ is non-abelian of order $p^3$,
$Z(N)=\langle c \rangle$, and 
$N$ is isomorphic to a dihedral group of order 8 when $p=2$, and it has exponent $p$ when $p$ is odd. In particular
 $a\in C_T(c) \setminus ( C_T(c)\cap  C_T(b))$,  $b\in (C_T(c)\cap  C_T(b)) \setminus ( C_T(c)\cap  C_T(b) \cap C_T(a))$ and  $c\in ( C_T(c)\cap  C_T(b) \cap C_T(a))$.
Thus $([1,1], [1,c], [1,b], [1,a])$ is an irredundant sequence that is not a base. By Proposition \ref{basediag} we have that $b(G)\le 4$, hence by Remark~\ref{rem:ibisirredundantsize} $G$ is not IBIS.
\end{proof}

\begin{prop}\label{PSU3q} Assume that $S^2 \le G\le W(2,S)$, with $T=\PSU(3,q)$. Then $G$ is not an IBIS group.\end{prop}
\begin{proof} Let  $q=p^f$, with $p$ prime, and let $\alpha$ be the automorphism of order $2$ of the field $\mathbb F_{q^2}$ with $q^2$ elements defined by $\alpha(x)=x^q$.
All $\alpha$-Hermitian forms in a vector space of dimension $3$ over $\mathbb F_{q^2}$ are equivalent, so we may assume that $\GU(3,q)$ is the isometry group of the form 
$B$ such that $B((x_1,x_2,x_3), (y_1,y_2,y_3))=x_1y_3^q+x_2y_3^q+x_3y_1^q$ associated to the matrix
 $$\begin{pmatrix} 0&0&1\\0&1&0\\1&0&0\end{pmatrix}.$$ Let
$$M:=\left\{ \begin{pmatrix} 1&-\alpha^{q}&\beta\\0&1&\alpha\\0&0&1\end{pmatrix}\big| \alpha,\beta \in \mathbb F_{q^2} ,\beta+\beta^q+\alpha \alpha^q=0\right\}.$$
It is not difficult to prove that $M$ is a subgroup of the special unitary group $\SU(3,q)$, and since $M$ intersects the center of $\SU(3,q)$ trivially, it can be identified with a subgroup of $T=\PSU(3,q).$
We shall use without further comment the fact that the trace map from $\mathbb{F}_{q^2}$ to $ \mathbb{F}_{q}$ given by $\mathrm{tr}(\beta)= \beta+\beta^q$ is surjective. 
Taking a non-zero element $\gamma\in \mathbb{F}_{q^2}$  such that $\gamma+\gamma^q=0,$ then
$$c:=\begin{pmatrix} 1&0&\gamma\\0&1&0\\0&0&1\end{pmatrix}$$
is a non-trivial central element of $M.$
Denote by $\omega$ a primitive element of $
\BBF_{q^2},$ then $\omega^{q+1}\in \BBF_q$ and consequently there exists $\delta\in \mathbb F_{q^2}$ such that $\delta+\delta^q=-\omega \omega^q.$ Hence we can consider the following elements of $M:$
$$a:=\begin{pmatrix} 1&-1&0\\0&1&1\\0&0&1\end{pmatrix}\; \mbox{and}\; \,b:=\begin{pmatrix} 1&-\omega^q &\delta \\0&1&\omega\\0&0&1\end{pmatrix}.$$
Since $\omega^q\ne \omega,$ then $a$ and $b$ do not commute.
Hence, we have that
 $a\in C(c)\setminus (C(c)\cap C(b))$, $b\in  (C(c)\cap C(b))\setminus (C(c)\cap C(b)\cap C(a)) $, and $c\in (C(c)\cap C(b)\cap C(a)) $. So $([1,1], [1,c], [1,b], [1,a])$ is an irredundant sequence  of size $4$ that is not a base. Now, the result follows combining Propoition~\ref{basediag} and Remark~\ref{rem:ibisirredundantsize}.
\end{proof}

So far, we proved that when $G$ is monolithic primitive, then $T^2\le G\le W(2,T)$ is a non IBIS group. Hence we need to consider the case $k\ge 3$.

\begin{prop} \label{prop:k=3}
If $T^3 \le G\le W(3,T)$ then $G$ is not IBIS.
\end{prop}
\begin{proof}
Assume by contradiction that $G$ is an IBIS group.
By Proposition~\ref{basediag} and Lemma~\ref{lem:notibis2} we can assume 
$P_G=\Sym(3)$. In particular $\sigma:=(1,2)\in P_G,$ so there exist $y\in \aut(T)$, $s_2,s_3\in T$ such that $(y,s_2y,s_3y)\sigma\in G.$ That is $(y,y,y)\sigma\in G$ for some $y\in \Aut(T).$
It follows from the classification of the finite simple groups that there exists  a non-trivial element $t$ in $C_T (y)$ (see for example (\cite[1.48]{gor}).
So we have that $(y,y,y)\sigma\in G_{([1,1,1], [1,1,t])}$, while $(y,y,y)\sigma\notin G_{[1,t,t]}$. If $s\notin C_T(t),$ then $(s,s,s)\in G_{[1,1,1]}\setminus G_{[1,1,t]}.$
Moreover $(t,t,t)\in G_{([1,1,1], [1,1,t], [1,t,t])}$. Hence $([1,1,1], [1,1,t], [1,t,t])$ is an irredundant sequence that is not a base.  In particular $G$ has an irredundant base of size at least $4$. Now, the result follows combining~\ref{basediag} and Remark~\ref{rem:ibisirredundantsize}.
\end{proof}

\begin{prop}\label{prop:kalmeno4}
Let $k\ge 4$ and let $T^k \le G\le W(k,T).$ Then $G$ is not an IBIS group.
\end{prop}
\begin{proof} If $P_G$ does not contain $\alt(k)$, then $b(G)=2$ by Proposition~\ref{basediag}, hence the result follows from Lemma~\ref{lem:notibis2}.
	So we can assume that $\alt(k)\le P_G.$
	Let  $(u_1,\dots, u_k) \in T^k \leq G.$ Recall that $g=(y,\dots,y)\sigma$ stabilises $[u_1,\dots, u_k]$ if and only if 
$${g}^{(u_1,u_2, u_3\dots, u_k)^{-1}}= (u_1yu_{1\sigma}^{-1},\dots,u_k y u_{k\sigma}^{-1})\sigma\in D(k,T),$$
i.e. if and only if
\begin{equation}\label{crit}u_1yu_{1\sigma}^{-1}=\dots=u_k y u_{k\sigma}^{-1}.
\end{equation} Consider in particular $(u_1,\dots,u_k)=
(t_1,t_2,1,\dots,1)$ where $t_1$ and $t_2$ are chosen so that they are not conjugated in $\aut(T)$ and $T=\langle t_1,t_2\rangle$ (choose for instance an involution $t_1$, then the existence of $t_2$ follows from the results in \cite{GK}). Let $\tau:=[u_1,\dots,u_k].$ We are going to deduce from (\ref{crit}) that $g=(y,\dots,y)\sigma$ stabilises $\tau$ if and only if
1 and 2 are fixed by $\sigma$ and $y=1.$ Let $w_i:=u_iyu_{i\sigma^{-1}}^{-1}$ with $u_i=t_i$ if $i\leq 2,$ $u_i=1$ otherwise. Assume $1\sigma^{-1}=j>2.$ Then $w_j=yt_1^{-1},$ and $w_i\in \{y,yt_2^{-1}\},$ when $i\geq 3$ and $i\neq j$. This would imply $w_i\neq w_j$ in contradiction with (\ref{crit}). So we must have $1\sigma^{-1} \in \{1,2\}$. A similar argument shows that $2\sigma^{-1} \in \{1,2\}.$ If follows  that $w_1=w_2=w_3=\dots=w_k=y.$ If $1\sigma=2$, then
$t_1yt_2^{-1}=y,$ i.e. $t_2=t_1^y$ in contradiction with the assumption that $t_1$ and $t_2$ are not conjugated in $\aut(T).$ So we must have $1\sigma=1$ and $2\sigma=2.$ This implies	$t_1yt_1^{-1} = t_2yt_2^{-1}=y,$ hence $y\in C_{\aut(T)}(\langle t_1,t_2\rangle)=C_{\aut(T)}(T)=1$.

Let $\iota:=[(1,\dots,1)]$. We have proved that $G_{(\iota,\tau)}=\{\sigma \in G \mid 1\sigma=1, 2\sigma=2\}.$ Since $\alt(k) \leq P_G,$ there exists $a \in \aut(T)$ such that $z=(a,\dots,a)(1,2)(3,4)\in G.$ Let
$1\neq s \in C_T(a)$ and consider $\eta:=[s,s,1,\dots,1].$
Clearly $G_{(\iota,\tau)}\leq G_{(\iota,\eta)}$
and $z\in G_{(\iota,\eta)} \setminus G_{(\iota,\tau)}.$ 
%
%
%
%
%
%
Further, there exists $b \in \aut(T)$ such that $v:=(b,\dots,b)(1,2,3) \in G.$ If follows from (\ref{crit}) and the fact that $s\neq 1,$ that  $v \in 
G_\iota \setminus G_\eta$.  So we have
$G_\iota > G_{(\iota,\eta)} > G_{(\iota,\eta,\tau) }=G_{(\iota,\tau)}$ and therefore $G$ is not IBIS.\end{proof}

Combining Propositions~\ref{cor:generatoribelli},~\ref{prop:A7},~\ref{prop:A5} ,~\ref{bPSL2q},~\ref{PSL},~\ref{PSL3q},~\ref{PSU3q},\ref{prop:k=3},~\ref{prop:kalmeno4}, we finally deduce the following result.
\begin{thm}\label{thm:nonmonolithicdiag}
Every monolithic primitive group of diagonal type is a non IBIS group.
\end{thm}

%
%

\section{Product type}\label{sec:prod}

Let $H\le \mathrm{Sym}(\Gamma)$ be a primitive group of almost simple type or of diagonal type with socle $T$. 
For $k\ge 2,$ we consider $W=H\wr \mathrm{Sym}(k)$ acting on $\Omega=\Gamma^k$, with its natural product action. Then $G$ is of product type if $T^k\le G\le W$ and the group
$P$ of the elements $\sigma \in \mathrm{Sym}(k)$ such that  $(h_1,\dots, h_k)\sigma \in G$ for some $h_i\in H$ is transitive. Write $\Omega=\Gamma_1\times \dots \times \Gamma_k$ where $\Gamma_i = \Gamma$ for each $i$.
Without loss of generality, we may assume that $G$ induces $H$ on each
of the $k$ factors $\Gamma_i$ of $\Omega$. In the whole section we will assume that $G$ is a primitive permutation group of product type.

%

\begin{thm}
The group $G$ is not an IBIS group. 
\end{thm}
\begin{proof}
To show this result, let assume by contradiction that $G$ is an IBIS group. 
We are going to construct two irredundant bases of $G$ and we will compare their lengths. 

%
Let $\SS:=\{\gamma_0, \gamma_1,\dots,\gamma_r\}$ be an irredundant base of $T\leq \sym(\Omega)$.
By Lemma~\ref{lem:notibis2} 
we deduce that
\begin{equation}\label{eq:rgrande}
r\ge2.
\end{equation}
Let ${\alpha}_i=(\gamma_i,\dots, \gamma_i)$ with $0\le i\le r$ and let $\beta_{i,j}$ be the $k$-tuple having the $i$-th component equal to $\gamma_j$ and the others equal to $\gamma_0$, for $1\le i\le k$ and $1\le j\le r$. 
The following sequences
\begin{align*}
\BB_s=\{\alpha_0, \alpha_1,\dots, \alpha_r\}\;\,\; \mbox{and} \,\;\; \BB_l=\{\alpha_0,\beta_{1,1},\dots, \beta_{1,r},\beta_{2,1}, \dots, \beta_{2,r}, \dots ,\beta_{k,1}, \dots, \beta_{k,r}\}
\end{align*}
are irredundant sequences of $G$, and $\BB_s$ has size $r+1,$ while $\BB_l$ has size $1+kr$. 
Indeed, from the choice of $\SS,$ for every $1\le a\le r-1$, there exists $u_a\in T$ that stabilizes $\gamma_0,\dots,\gamma_a$ but does not stabilize $\gamma_{a+1},$ and consequently $(u_a, u_a \dots,u_a )\in G$ stabilises $\alpha_0, \alpha_{1},\dots, \alpha_a$ but not  $\alpha_{a+1}$. Moreover $(u_a, 1, \dots,1 )\in G$ stabilises $\alpha_0, \beta_{1,1}, \dots, \beta_{1,a}$ but not  $\beta_{1,a+1}$,
$(1, u_a, \dots,1 )\in G$ stabilises $\alpha_0, \beta_{1,1}, \dots, \beta_{1,r}, \beta_{2,1}, \dots, \beta_{2,a} $ but not $ \beta_{2,a+1}$ and proceeding with this argument we finally get that
$(1, 1, \dots,u_a )\in G$ stabilises $\alpha_0, \beta_{1,1}, \dots, \beta_{1,r}, \beta_{2,1}, \dots, \beta_{2,r}, \dots, \beta_{k,1}, \dots, \beta_{k,a}$, but not  $\beta_{k,a+1}$.

\medskip 

Let $X:=G\cap H^k$.
It is not difficult to observe that $$X_{(\BB_s)}=X_{(\BB_l)}=\{(h_1,\dots,h_k)\in X\mid h_1 \in H_{(\SS)}, \dots, h_k \in H_{(\SS)}\}.$$
Now, let $\AA=\{\delta_1, \dots, \delta_b\}$ be a base of minimal size of $X_{(\BB_s)}=X_{(\BB_l)}$. Hence $\BB_s\cup \AA$ and $\BB_l\cup \AA$ are still irredundant sequences for $G$.

Here, let $\delta_1,\dots, \delta_c,$ and $\nu_1,\dots, \nu_{c'}$ be elements of $G$ such that the sequences $\BB_1=\BB_s\cup \AA \cup \{\delta_1,\dots, \delta_c\}$ and  $\BB_2=\BB_l\cup A \cup \{ \nu_1,\dots, \nu_{c'} \}$ are irredundant bases of $G$.
Since $G$ is IBIS, then $|\BB_1|=r+1+b+c=1+rk+ b+c'=|\BB_2|.$ 
Hence \begin{align}\label{eq:firstbit}r(k-1)=c-c'\le c.\end{align}
Note that $G_{(\BB_s\cup \AA)}\cap X=1,$ then $G_{(\BB_s\cup \AA)}\cong  G_{(\BB_s\cup \AA)}/(G_{(\BB_s\cup \AA)}\cap X) \cong  G_{(\BB_s\cup \AA)}X / X \le \sym(k)$.
So, denoting by $\ell(\sym(k))$ the maximal length of a chain of subgroups of $\sym(k),$ we can estimate $c$ with $\ell(\sym(k)).$ From (\ref{eq:firstbit}) and (\ref{eq:rgrande})  we deduce that
\begin{equation}\label{eq:finale}
2(k-1)\le \ell(\sym(k)).
\end{equation}
By~\cite{cameronsoltur} we have that $\ell(\sym(k))< 3k/2,$ hence (\ref{eq:finale}) yields $k<4$.
Since $\ell(\sym(2))=1$ and  $\ell(\sym(3))=2$, (\ref{eq:finale}) allows us to exclude also $k=2$ and $k=3$. This contradicts the assumption $k\ge 2$ and the result follows.
 \end{proof}

 \section{Twisted wreath product type} \label{sec:TW}
 
Let $T$ be a non-abelian simple group, and let $k$ be an integer that is at least $2.$ A
group $G$ of twisted wreath product type with socle $T^k$ acts primitively on a set $|\Omega|$ 
with degree $|T^k|$ and is also a twisted wreath product, which is a split extension of $T^k$ by a transitive
subgroup $P$ of $\sym(k)$. In the whole section we will assume that $G$ is a primitive permutation group of twisted wreath product type.

\begin{thm}
The group $G$ is not an IBIS group. 
\end{thm}
\begin{proof}We may identify $G$ with a subgroup of the  wreath product $W=\aut(S) \wr P$, with the property that
	$\soc(G)=T^k$ and $T^k$ has a maximal complement, say $H$, in $G$, which is isomorphic to $P.$ In particular for every $\sigma \in P$ there exists a unique element $(a_1,\dots,a_k) \in (\aut(T))^k$ such that
	$(a_1,\dots,a_k)\sigma\in H.$ 
	
By \cite[Theorem 5.0.1]{fawcettetesi}, if $P$ is a primitive subgroup of $\sym(k),$ then $b(G)=2$.  In this case, since $\soc(G)$ is non-abelian, the result follows from Lemma~\ref{lem:notibis2}. So we may assume that $k$ is not a prime. Since $P$ is transitive, it contains a fixed-point free permutation $\sigma.$ Either $\sigma$ or one of its powers (in the case when $\sigma$ is a $k$-cycle) admits at least two orbits of size at least 2. So it is not restrictive to assume that $P$ contains an element 
\begin{align}\label{eq:hp1}
\rho=(1,\dots , r)(r+1,r+2,\dots , r+s) \tau 
\end{align}
for some  $r, s\ge 2$ and $\tau \in \sym(k)$ which pointwise fixes the set $\{1,\dots,r+s\}$. There exists
a unique element $(a_1,\dots,a_k)\in (\aut(T))^k$ such that
$h:=(a_1,\dots,a_k)\rho \in H.$ It follows from the classification of the finite simple groups that $C_T(a_1\cdots a_r)\neq 1$ (\cite[1.48]{gor}). In particular there exists $t\in C_T(a_1\cdots a_r)$ of prime order, say $p.$ Choose $u \in T$ of prime order $q$, with $q\neq p,$ and consider the following elements of $T^k$: $y_1=(t,t^{a_1},t^{a_1a_2},\dots,t^{a_1a_2\cdots a_{r-1}},1,\dots,1),$
$y_2=(1,\dots,1,u,1, \dots ,1),$ where $u$ is the entry in position $r+1,$ and $y=y_1y_2.$ We have that $h \in C_H(y_1)=C_H(y^q)$ but $h \notin C_H(y)$.
Hence $H\cap H^{y^q}\cap H^y=H\cap H^y=C_H(y) < C_H(y^q)=H \cap H^{y^q}<H$ and therefore $G$ is not IBIS.
%
\end{proof}

\section{Proof of Theorem~\ref{thm:main}}\label{sec:mainthm}
  We work through the non-affine and non-almost simple classes of the O'Nan-Scott Theorem. If $G$ is of diagonal type, as showed in Propositions~\ref{prop:sorpresa},~\ref{prop:solosorpresa0},~\ref{prop:solosorpresa} and Theorem~\ref{thm:nonmonolithicdiag}, there is a unique infinite family of IBIS group. 
If $G$ is of product action type or twisted wreath product,
then the result follows immediately from
Section~\ref{sec:prod} and Sections~\ref{sec:TW} respectively. Hence the proof is complete.

\bigskip

\noindent

ACKNOWLEDGMENTS. All authors are members of GNSAGA.

	\end{document}